\documentclass{amsart}

\title[On the expansion of certain vector-valued characters of $U_q(\mathfrak{gl}_n)$]{On the expansion of certain vector-valued characters of $U_q(\mathfrak{gl}_n)$ with respect to the Gelfand-Tsetlin basis}
\author{Vidya Venkateswaran}
\address{Department of Mathematics, MIT, Cambridge, MA 02139}
\email{\href{mailto:vidyav@math.mit.edu}{vidyav@math.mit.edu}}
\thanks{Research supported by NSF Mathematical Sciences Postdoctoral Research Fellowship DMS-1204900}
\subjclass[2000]{}

\usepackage{hyperref}
\usepackage[svgnames]{xcolor}
\hypersetup{colorlinks,breaklinks,
            linkcolor=DarkBlue,urlcolor=DarkBlue,
            anchorcolor=DarkBlue,citecolor=DarkBlue}
\usepackage{url}
\usepackage{euscript}
\usepackage[verbose,letterpaper,tmargin=1in,bmargin=1in,lmargin=1.5in,rmargin=1.5in]{geometry}

\usepackage{amssymb}
\usepackage{amsmath,amsthm}
\usepackage{mathtools}
\usepackage{xypic}
\usepackage{verbatim}

 \DeclareMathOperator{\Tr}{Tr}
 \DeclareMathOperator{\Proj}{Proj}

\newtheorem{theorem}{Theorem}[section]

\newtheorem{lemma}[theorem]{Lemma}
\newtheorem{definition}[theorem]{Definition}
\newtheorem{proposition}[theorem]{Proposition}
\newtheorem{corollary}[theorem]{Corollary}
\newtheorem*{remark}{Remark}

\newtheorem*{remarks}{Remarks}

\newcommand{\N}{{\mathbb N}}

\newcommand{\Z}{{\mathbb Z}}
\newcommand{\C}{{\mathbb C}}

\begin{document}

\begin{abstract}
Macdonald polynomials are an important class of symmetric functions, with connections to many different fields.  Etingof and Kirillov showed an intimate connection between these functions and representation theory: they proved that Macdonald polynomials arise as (suitably normalized) vector-valued characters of irreducible representations of quantum groups.  In this paper, we provide a branching rule for these characters.  The coefficients are expressed in terms of skew Macdonald polynomials with plethystic substitutions.  We use our branching rule to give an expansion of the characters with respect to the Gelfand-Tsetlin basis.  Finally, we study in detail the $q=0$ case, where the coefficients factor nicely, and have an interpretation in terms of certain $p$-adic counts.
\end{abstract}

\maketitle

\section{Introduction}
Macdonald polynomials were originally discovered in the 1980s \cite{Mac, MacP}, and have found a variety of
uses in mathematics, appearing in a number of disparate fields (mathematical physics, combinatorics,
representation theory and number theory, among others).  These polynomials have the key property of
being invariant under all permutations of their $n$ variables.  They are indexed by partitions
$\lambda$ with length at most $n$, and form an orthogonal basis for the ring of symmetric
polynomials with coefficients in $\mathbb{C}(q,t)$ with respect to a certain density function.  The
existence of such polynomials was proved by exhibiting particular difference operators which have
these polynomials as their eigenfunctions.  Macdonald polynomials contain many important families as
particular degenerations of the parameters $q$ and $t$.  In particular, the ubiquitous Schur functions are obtained
by setting $q=t$; crucially, these are characters of irreducible representations of $GL_{n}$.
Hall-Littlewood polynomials are recovered in the limit $q=0$, and these have interpretations
as zonal spherical functions on $p$-adic groups.  Some other important subfamilies are the monomial,
elementary, and power sum symmetric functions.  

Given the various connections to representation theory, one might ask whether Macdonald polynomials
arise as characters of certain irreducible representations.  Etingof and Kirillov discovered such a
realization in \cite{EK}, where they demonstrate that Macdonald polynomials are ratios of
vector-valued characters of representations of the quantum group $U_q(\mathfrak{gl}_n)$.  Recall
that the finite-dimensional, irreducible representations $V_\lambda$ of $U_g(\mathfrak{gl}_n)$ are
indexed by $\lambda \in \mathcal P_+^{(n)} = \{(\lambda_1, \dotsc,\lambda_n): \lambda_{i} -
\lambda_{i+1} \in \Z_+\}$.  Note that elements of $\mathcal P_+^{(n)}$ can be written as
$(a,a,\dotsc,a) +
\widetilde \lambda$, where $a \in \C$ and $\widetilde \lambda$ is a standard partition of length
$n$.  Let $k \in \N$ be fixed, then they show the existence of an intertwining operator (unique up
to scaling):

\begin{equation}\label{eq:intertwiner}
  \phi^{(k)}_{\lambda}: V_{\lambda + (k-1)\rho} \rightarrow V_{\lambda + (k-1)\rho} \otimes U,
\end{equation}
where $U \simeq V_{(k-1)\cdot(n-1,-1,\dotsc,-1)}$ and $\rho =
\bigl(\frac{n-1}{2},\frac{n-3}{2},\dotsc,\frac{1-n}{2}\bigr)$.  Note that $U$ has the special property that all weight subspaces are one-dimensional.  Fix the normalization of
$\phi_{\lambda}^{(k)}$ so that $v_{\lambda + (k-1)\rho} \rightarrow v_{\lambda + (k-1)\rho} \otimes u_{0}
+ \cdots$, where $v_{\lambda + (k-1)\rho}$ is a fixed non-zero highest weight vector for $V_{\lambda
+ (k-1)\rho}$ and $u_{0}$ is a fixed non-zero vector in the (one-dimensional) weight zero subspace
of $U$.  Consider the corresponding trace function of this operator: 

\begin{equation*}
  \Phi_{\lambda}^{(k)}(x_1,\dotsc,x_n) = \Tr(\phi^{(k)}_{\lambda} \cdot x^{h}) \in
  \C(q)[x_1,\dotsc,x_n],
\end{equation*}
where $q^h = q^{(h_1 + \dotsb + h_n)} \in U_q(\mathfrak{gl}_n)$.  Then Etingof and Kirillov proved the following intimate connection between these trace functions and Macdonald polynomials:

 \begin{theorem} \label{EKfin} \cite{EK}
    The Macdonald polynomial $P_{\lambda}(x; q^{2}, q^{2k})$ is given by the ratio 
     \begin{equation*}
       P_\lambda(x; q^2, q^{2k}) = \frac{\Phi_{\lambda}^{(k)}(x)}{\Phi_{0}^{(k)}(x)}.
     \end{equation*}
   Moreover, there are formal power series $\widetilde{\Phi}_\lambda(x; q, t) \in
     \C(q,t)[[x_1,\dotsc,x_n]]$ such that
     \begin{equation*}
     P_\lambda(x; q^2, t^2) = \frac{\widetilde{\Phi}_{\lambda}(x; q, t)}{\widetilde{\Phi}_{0}(x;
     q, t)},
   \end{equation*}
   and
   \begin{equation*}
     \Phi_{\lambda}^{(k)}(x_1,\dotsc,x_n) = \widetilde{\Phi}_{\lambda}(x_1,\dotsc,x_n; q, q^k).
   \end{equation*}
 \end{theorem}

The two-parameter family $\widetilde{\Phi}_\lambda(x;q,t)$ can be realized as
the trace function of an intertwining operator analogous to
\eqref{eq:intertwiner}, where $V_{\lambda + (k-1)\rho}$ and $U$ are replaced by
suitable infinite-dimensional, irreducible representations of $\C(t)
\otimes U_q(\mathfrak{gl}_n)$ (see \cite{EK} or Section \ref{sec:verma}
for details).

In this paper we consider the expansion of the trace function $\Phi^{(k)}(x)$ with respect to the
Gelfand-Tsetlin basis of $V_{\lambda + (k-1)\rho}$.  We give explicit combinatorial formulas for the
diagonal coefficients of the intertwining operator $\phi^{(k)}$ with respect to this basis, as sums
of products of well-known rational functions appearing in symmetric function theory, specialized to
$t=q^k$.  Inspired by Kashiwara's theory of crystal bases \cite{Kash}, we consider the $q \to
0$ limit of these coefficients (with algebraically independent $t$).  We find that for $t = p^{-1}$
with $p$ an odd prime, this limit is proportional to a count of certain chains of groups of $p$-adic
type.  While it is well-known that Hall-Littlewood polynomials are intimately related to $p$-adic
representation theory, it is somewhat surprising to find $p$-adic quantities arising in the context
of quantum groups.

We will now state our results more precisely.  First recall that a Gelfand-Tsetlin pattern of shape
$\lambda \in \mathcal P_+^{(n)}$ is a sequence $\lambda = \lambda^{(0)}\succeq \lambda^{(1)} \succeq
\dotsb \succeq \lambda^{(n-1)}$, where $\lambda^{(i)} \in \mathcal P_+^{(n-i)}$ and $\succeq$
denotes the interlacing relation: $\lambda^{(i+1)}_{j} - \lambda^{(i)}_{j+1} \in \Z_+, \;
\lambda^{(i)}_{j+1} - \lambda^{(i+1)}_{j+1} \in \Z_+$.  This may be visualized as an array consisting of $n$ rows with the parts of $\lambda$ in the first row, parts of $\lambda^{(1)}$ in the second row, etc.  There is a canonical basis of $V_\lambda$
which is indexed by $GT(\lambda)$, the Gelfand-Tsetlin patterns of shape $\lambda$.

Our aim is to compute the expansion of the trace function $\Phi^{(k)}$ in the Gelfand-Tsetlin basis
for $V_{\lambda + (k-1)\rho_n}$.  We would like to index these patterns in a uniform way with
respect to the parameter $k \in \N$.  Conveniently, there is a canonical way of doing this for the
Gelfand-Tsetlin patterns whose coefficient in the expansion of $\Phi^{(k)}$ is non-zero:

\begin{definition}\label{def:gt_index}
  Let $\lambda \in \mathcal P^{(n)}_+$, and let $\lambda = \mu^{(0)} \supset \mu^{(1)} \supset \dotsb
\supset \mu^{(n-1)}$ be such that:
\begin{enumerate}
  \item $\mu^{(i)} \in \mathcal P^{(n-i)}_+$. 
  \item $\mu^{(i)}_j - \mu^{(i+1)}_{j} \in \Z_+$, $1 \leq j \leq n-i-1$.
  \item $\mu^{(i)}_j - \mu^{(i+1)}_{j-1} \leq k-1$, $2 \leq j \leq n-i$.
\end{enumerate}
    
Define $\overline{\mu}^{(i)} = \mu^{(i)} + (k-1)\rho_{n-i} + (k-1)\cdot(i/2,\dotsc,i/2)$, then
$(\overline{\mu}^{(0)} \succeq \overline{\mu}^{(1)} \succeq \dotsb \succeq
\overline{\mu}^{(n-1)})$ is a Gelfand-Tsetlin pattern of shape $\lambda + (k-1)\rho_n$.
\end{definition}

We can now give our formula for the coefficients in the expansion of $\Phi^{(k)}_\lambda(x)$ in the
Gelfand-Tsetlin basis, along with a new branching formula for the functions $\widetilde{\Phi}(x;q,t)$.  This will be expressed in terms of functions $\psi_{\gamma/\delta}, \Omega_{\gamma/\delta}$, which are plethystic substitutions of skew Macdonald polynomials, and $d_{\alpha}$, which is the norm with respect to a particular inner product, see the Background section for more details.
\begin{theorem} \label{qtthm}
  For $\mu \subset \lambda$ with $\lambda \in \mathcal P_+^{(n)}$, $\mu \in \mathcal P_+^{(n-1)}$,
  define
  \begin{equation} \label{eq:branchcoeff}
 c_{\lambda, \mu}(q,t) = d_{\mu}(q^{2}, t^{2}) \sum_{\substack{\beta \in \mathcal P_+^{(n-1)} \\
 \mu \subseteq \beta \preceq \lambda}} \frac{\psi_{\lambda/\beta}(q^{2},
 t^{2})}{d_{\beta}(q^{2}, t^{2})} \Omega_{\beta/\mu}(q^{2}, t^{2}).
 \end{equation}
Then the trace functions $\widetilde{\Phi}(x;q,t)$ satisfy the branching rule:
\begin{equation*}
  \widetilde\Phi_{\lambda}^{(n)}(x;q,t) = \sum_{\substack{\mu \in \mathcal P_+^{(n+1)} \\ \mu
\subset \lambda}} c_{\lambda, \mu}(q,t) \cdot
 \widetilde{\Phi}_{\mu}^{(n-1)}(x;q,t) \cdot x_{n}^{\rho(\lambda,
 \mu)}, 
\end{equation*}
where $\displaystyle \rho(\lambda, \mu) = \sum_{i=1}^{n-1} \lambda_{i} - \mu_{i}$.

Moreover, with respect to the Gelfand-Tsetlin basis of $V_{\lambda + (k-1)\rho_n}$, the diagonal
coefficient of the intertwining operator $\phi^{(k)}$ corresponding to the Gelfand-Tsetlin pattern
$\Lambda = (\lambda^{(0)} \succeq \dotsb \succeq \lambda^{(n-1)}) \in GT(\lambda + (k-1)\rho_n)$ is equal to  
 \begin{equation*}
  c_{\Lambda}(q, q^{k}) = \begin{cases}
    \displaystyle\prod_{1 \leq i \leq n-1}c_{\overline{\mu}^{(i-1)},
    \overline{\mu}^{(i)}}(q,q^{k}), & \exists\; (\mu^{(0)} \supseteq \dotsb
    \supseteq \mu^{(n-1)}) \text{ s.t. } \lambda^{(i)} = \overline{\mu}^{(i)}\\
    0, & \text{otherwise}
  \end{cases}.
 \end{equation*}
\end{theorem}  

Kashiwara's crystal bases \cite{Kash} allow one to interpret finite-dimensional representations
of $U_q(\mathfrak{gl}_n)$ in the ``crystal limit'' $q \to 0$.  Remarkably, there is a rich
combinatorial structure in the crystal limit.  Since the Gelfand-Tsetlin basis yields a crystal
basis for $V_\lambda$, it seems natural to consider the limit as $q \to 0$ of the coefficients in
Theorem \ref{qtthm}.  A priori this limit need not even exist, but in fact we are able to obtain a simple
closed formula, in a factorized form:

\begin{theorem} \label{trestr}
   Let $\mu \subset \lambda$ with $\lambda \in \mathcal P_+^{(n)}$, $\mu \in \mathcal P_+^{(n-1)}$,
   then we have
   \begin{equation} \label{eq:hlcoeffs}
   \lim_{q \to 0} c_{\lambda, \mu}(q,t) = \frac{b_{\mu}(t^{2})}{b_{\lambda}(t^{2})}
 (1-t^{2})sk_{\lambda/\mu}(t^{2}) = t^{2\sum_{j}  \binom{\lambda_{j}'-\mu_{j}'}{2}}
 \prod_{j \geq 1} \binom{\lambda_{j}'-\mu_{j+1}'}{\lambda_{j}'-\lambda_{j+1}'}_{t^{2}}.
 \end{equation}
\end{theorem}

\noindent Here the coefficients $sk_{\lambda/\mu}(t)$ are those studied in \cite{W, KL} in the context of Pieri rules.  Note that when $t = p^{-1}$ for an odd prime $p$, 
\begin{equation*}
 sk_{\lambda/\mu}(t) = t^{n(\lambda) - n(\mu)} \alpha_{\lambda}(\mu; p),
 \end{equation*}
where $\alpha_{\lambda}(\mu; p)$ is the number of subgroups of type $\mu$ in a finite abelian $p$-group of type $\lambda$.  For $S = (\mu^{(0)} \supset
    \mu^{(1)} \supset \dotsb \supset \mu^{(n-1)})$ with $\mu^{(i)} \in \mathcal P_+^{(n-i)}$, we define the coefficient
\begin{equation} \label{skScoeff}
sk_{S}(t) = sk_{\mu^{(0)}/\mu^{(1)}}(t) sk_{\mu^{(1)}/\mu^{(2)}}(t) \cdots sk_{\mu^{(n-2)}/\mu^{(n-1)}}(t).   
\end{equation}
Note that when $t = p^{-1}$, $sk_{S}(t)$ is (up to a power of $t$) the number of nested chains of subgroups with types specified by the sequence $S$.  We also let 
\begin{equation*}
wt(S) = \big(\rho(\mu^{(n-1)}, 0), \rho(\mu^{(n-2)}, \mu^{(n-1)}), \cdots, \rho(\mu^{(1)}, \mu^{(2)}), \rho(\mu^{(0)}, \mu^{(1)})\big).
\end{equation*}

\begin{theorem} \label{thm:gt_padic}
  Let $\lambda \in \mathcal P_+^{(n)}$. Then
  \begin{equation} \label{eq:hl_tracefct}
    \lim_{q \to 0} \widetilde{\Phi}_\lambda(x; q, t) = \frac{(1-t^{2})^{n}}{b_{\lambda}(t^{2})} \sum_{\substack{S = (\lambda = \mu^{(0)} \supset
    \mu^{(1)} \supset \dotsb \supset \mu^{(n-1)}) \\ \mu^{(i)} \in \mathcal P_+^{(n-i)}}} sk_{S}(t^{2})x^{wt(S)}.
  \end{equation}
\end{theorem}

Note that the coefficients for the Gelfand-Tsetlin basis of $V_{\lambda +
(k-1)\rho_n}$ are obtained by specializing $t=q^k$ in \eqref{eq:branchcoeff},
and hence we are only able to obtain the $t=0$ specialization of
\eqref{eq:hlcoeffs} in the crystal limit.  As mentioned above, there is a
representation theoretic realization of $\widetilde \Phi(x; q, t)$, with $t$
algebraically independent from $q$, as the trace function of an intertwiner
between infinite-dimensional modules over $\C(t) \otimes
U_q(\mathfrak{gl}_n)$.  There is an analogue of the Gelfand-Tsetlin basis for
these modules, and we can obtain \eqref{eq:hlcoeffs} for general $t$ as the $q
\to 0$ limit of the coefficients in the expansion of $\widetilde \Phi(x; q, t)$
with respect to this basis (see Section \ref{sec:verma} within the paper for
more details about this).  Unfortunately these modules do not fit into
Kashiwara's framework, and so we have not been able to find a direct
interpretation of \eqref{eq:hlcoeffs} in terms of crystal bases.  However, the
simple combinatorial structure of our formula in the limit $q \to 0$ does seem
to suggest a possible connection, and we leave it as an open question to
describe this connection more precisely.

\medskip
\noindent\textbf{Acknowledgements.} The author would like to thank Pavel Etingof for suggesting this work, and for many helpful discussions and comments.  She would also like to thank Eric Rains and Ole Warnaar for helpful comments.

\section{Background on symmetric function theory}

Recall that $\lambda = (\lambda_{1}, \dots, \lambda_{n}) \in (\mathbb{Z}_{+})^{n}$ is a partition if $\lambda_{i} \geq \lambda_{i+1}$.  There is a partial order on partitions defined by $\lambda > \mu$ if and only if $\sum \lambda_{i} = \sum \mu_{i}$ and for some $k<n$ we have $\lambda_{i} = \mu_{i}$ for all $i \leq k$ and $\lambda_{k+1} > \mu_{k+1}$.  We will work with polynomials of $n$ variables, i.e., over $\mathbb{C}[x_{1}, \dots, x_{n}]$.  For $\lambda \in \mathbb{Z}^{n}$, we let $x^{\lambda} = x_{1}^{\lambda_{1}} \cdots x_{n}^{\lambda_{n}}$.

We fix $k \in \mathbb{N}$, and set $t = q^{k}$.  Let $\rho = (\frac{n-1}{2}, \frac{n-3}{2}, \dots, \frac{1-n}{2})$ be half the sum of the positive roots; we will also write $\rho_{n}$ when it is not clear from context.  Note that 
\begin{multline} \label{rhoeq}
\rho_{n} - \rho_{n-1} = \Big(\frac{n-1}{2}, \frac{n-3}{2}, \dots, \frac{1-n}{2}\Big) - \Big(\frac{n-2}{2}, \frac{n-4}{2}, \dots, \frac{2-n}{2}\Big) 
\\ = \Big( \frac{1}{2}, \dots, \frac{1}{2}, \frac{1-n}{2} \Big).
\end{multline}

 We now define a number of different coefficients arising from symmetric function theory, see \cite{Mac}; we also review some relevant results from the literature.
 
 \begin{definition} \label{gfcn}
Define the functions $g(\gamma; q^{2}, t^{2})$ for $\gamma$ a partition by
\begin{equation*}
g(\gamma; q^{2},t^{2}) = t^{2|\gamma|} \frac{(t^{-2}q^{2};q^{2})_{\gamma_{1}} \cdots (t^{-2}q^{2};q^{2})_{\gamma_{n-1}}}{(q^{2};q^{2})_{\gamma_{1}} \cdots (q^{2};q^{2})_{\gamma_{n-1}}}.
\end{equation*}
\end{definition}

 \begin{definition} \label{mPieri}
 Let $c^{\delta}_{\gamma \mu}(q,t)$ be the coefficients in the following Pieri rule:
 \begin{equation*}
 m_{\gamma}^{(n)}(x) P_{\mu}^{(n)}(x;q,t) = \sum_{\delta} c_{\gamma \mu}^{\delta}(q,t) P_{\delta}^{(n)}(x;q,t).
 \end{equation*}
 \end{definition}
 We note that the coefficients $c_{\gamma \mu}^{\delta}$ can be determined via the change of basis coefficients $\{m_{\gamma}^{(n)}(x)\} \rightarrow \{P_{\eta}^{(n)}(x;q,t)\}$ in conjunction with the Pieri coefficients that express the product $P_{\eta}^{(n)}(x;q,t)P_{\mu}^{(n)}(x;q,t)$ in the Macdonald polynomial basis.

 \begin{definition}
 Let $\Omega_{\beta/\mu}(q^{2},q^{2k})$ be the coefficient on $P_{\beta}^{(n-1)}(x;q^{2}, q^{2k})$ in the expansion of
 \begin{equation*}
 P_{\mu}^{(n-1)}(x;q^{2}, q^{2k}) \prod_{i  = 1}^{n-1} \frac{(q^{2}x_{i};q^{2})_{\infty}}{(q^{2k}x_{i};q^{2})_{\infty}}
 \end{equation*} 
 in the basis $\{ P_{\beta}^{(n-1)}(x;q^{2}, q^{2k}) \}_{\beta}$.  
 \end{definition}

  We now recall the branching rule for Macdonald polynomials.
 \begin{theorem} \label{MDbr}
 \begin{equation*}
 P_{\lambda}^{(n)}(x;q,t) = \sum_{\mu \preceq \lambda} x_{n}^{|\lambda - \mu|} \psi_{\lambda/\mu}(q,t) P_{\mu}^{(n-1)}(x;q,t)
 \end{equation*}
  \end{theorem}
  \begin{proof}
  See (1.7) of \cite{LW} for example.
  \end{proof}
 
 \begin{remark}
  There is a product formula for the coefficients $\psi_{\lambda/\mu}(q,t)$ appearing above (\cite{Mac} p 342)
  \begin{equation*}
  \psi_{\lambda/\mu}(q,t) = \prod_{1 \leq i \leq j \leq l(\mu)} \frac{f(q^{\mu_{i}-\mu_{j}}t^{j-i})f(q^{\lambda_{i}-\lambda_{j+1}}t^{j-i})}{f(q^{\lambda_{i}-\mu_{j}}t^{j-i})f(q^{\mu_{i}-\lambda_{j+1}}t^{j-i})},
  \end{equation*}
 where $f(a) = (at)_{\infty}/(aq)_{\infty}$ with $(a)_{\infty} = \prod_{i \geq 0} (1-aq^{i})$.  
 \end{remark}
 
 \begin{proposition}
 We have
 \begin{equation*}
 \lim_{q \rightarrow 0} \psi_{\lambda/\mu}(q,t) = \prod_{\substack{\{j: \lambda_{j}' = \mu_{j}'\\\text{and } \lambda_{j+1}' = \mu_{j+1}' +1 \}}} (1-t^{m_{j}(\mu)})
 \end{equation*}
 if $\lambda/\mu$ is a horizontal strip, and zero otherwise.
 \end{proposition}
 \begin{proof}
 This follows from the branching rule for Hall-Littlewood polynomials (see for example \cite{Mac} p228 (5.5'), (5.14')).
 \end{proof}

 \begin{definition}
 Let
 \begin{equation*}
 \phi_{\lambda/\mu}(t) = \prod_{\substack{\{j: \lambda_{j}' = \mu_{j}'+1\\ \text{and } \lambda_{j+1}' = \mu_{j+1}'  \}}} (1-t^{m_{j}(\lambda)}),
 \end{equation*}
 if $\lambda/\mu$ is a horizontal strip, and zero otherwise.
 \end{definition}
 
 Note that these coefficients are the $q \rightarrow 0$ limiting case of $\phi_{\lambda/\mu}(q,t)$ which also arise as branching coefficients.
 
 \begin{remark}
The functions $\phi_{\lambda/\beta}(q,t), \Omega_{\beta/\bar{\mu}}(q,t)$ have interpretations in terms of skew Macdonald polynomials (in parameters $q,t$) with plethystic substitutions.  In particular, we have $\phi_{\lambda/\beta}(q,t) = Q_{\lambda/\beta}(1)$ and $\Omega_{\beta/\bar{\mu}}(q,t) = Q_{\beta/\bar{\mu}}\big( \frac{t-q}{1-t} \big) = t^{|\beta/\bar{\mu}|} Q_{\beta/\bar{\mu}}\big( \frac{1-q/t}{1-t} \big)$ (see for example \cite{W}), and both these quantities have nice factorized forms.
\end{remark}

 We will write $\psi_{\lambda/\mu}(t)$, $g(\gamma; t^{2})$, etc. to denote the limit $q \rightarrow 0$ of these functions.  
 
 \begin{definition} \cite{AK, KL}
 For any skew shape $\lambda/\mu$, define the coefficients
 \begin{equation*}
 sk_{\lambda/\mu}(t) = t^{\sum_{j} \binom{\lambda_{j}'-\mu_{j}'}{2}} \prod_{j \geq 1} \binom{\lambda_{j}'-\mu_{j+1}'}{m_{j}(\mu)}_{t}.
 \end{equation*}
 \end{definition}

 \begin{theorem} \label{Lauve-K} \cite{AK, KL}
 For a partition $\lambda$ and $r \geq 0$, we have
 \begin{equation*}
 P_{\lambda}^{(n)}(x;t)s_{r}^{(n)}(x) = \sum_{\lambda+} sk_{\lambda+/\lambda}(t) P_{\lambda+}^{(n)}(x;t),
 \end{equation*}
 with the sum over partitions $\lambda \subset \lambda+$ for which $|\lambda+/\lambda| = r$. 
  \end{theorem}

 We now recall two inner products that will appear throughout the paper.  We let $\langle \cdot, \cdot \rangle_{n}$ denote the Macdonald inner product (defined via integration over the $n$-torus).  In particular,
 \begin{multline} \label{integip}
 \langle P_{\lambda}^{(n)}(x;q,t), P_{\mu}^{(n)}(x;q,t) \rangle_{n} = \int_{T_{n}} P_{\lambda}^{(n)}(x;q,t) P_{\mu}^{(n)}(x^{-1};q,t) \tilde \Delta_{S}(x;q,t) dT \\= \delta_{\lambda, \mu} \frac{1}{d_{\lambda}(q,t)},
 \end{multline} 
 where an explicit formula for $d_{\lambda}(q,t)$ can be found in \cite{Mac}.  Also let $Q_{\mu}^{(n)}(x;q,t) = b_{\mu}(q,t)P_{\mu}^{(n)}(x;q,t)$ be scalar multiples of the Macdonald polynomials, and recall the other inner product $\langle \cdot, \cdot \rangle'$ which satisfies
\begin{equation} \label{otherip}
\langle P_{\lambda}^{(n)}(x;q,t), Q_{\mu}^{(n)}(x;q,t) \rangle' = \delta_{\lambda, \mu}
\end{equation} 
(so that $\langle P_{\lambda}^{(n)}(x;q,t), P_{\lambda}^{(n)}(x;q,t) \rangle' = \frac{1}{b_{\lambda}(q,t)}$). Note that this inner product is independent of $n$, and we have $\lim_{n \rightarrow \infty} \langle \cdot, \cdot \rangle_{n} = \langle \cdot, \cdot \rangle'$.  We have
 \begin{equation*}
 b_{\lambda}(t) = \prod_{i \geq 1} \phi_{m_{i}(\lambda)}(t),
 \end{equation*}
 where $m_{i}(\lambda)$ denotes the number of times $i$ occurs as a part of $\lambda$ and 
 \begin{equation*}
 \phi_{r}(t) = (1-t)(1-t^{2}) \cdots (1-t^{r}).
 \end{equation*}
 We also have
 \begin{equation*}
 d_{\lambda}(t) = \frac{1}{(1-t)^{n}} \prod_{i \geq 0} \phi_{m_{i}(\lambda)}(t),
 \end{equation*}
 so that if $l(\lambda) = n$, $b_{\lambda}(t) (1-t)^{n} = d_{\lambda}(t)$.
 
 We recall the following fact relating the branching coefficients $\phi_{\lambda/\beta}$ and $\psi_{\lambda/\beta}$ \cite{Mac}.
 \begin{proposition}\label{phipsi}We have
 \begin{equation*} 
 \phi_{\lambda/\beta}(q,t)/b_{\lambda}(q,t) = \psi_{\lambda/\beta}(q,t)/b_{\beta}(q,t).
 \end{equation*}
 \end{proposition}

  Note that, using (\ref{integip}) and (\ref{otherip}), the coefficients $c^{\delta}_{\lambda \mu}$ and $sk_{\lambda+/\lambda}$ may be defined in terms of inner products.  We have
 \begin{multline*}
c^{\delta}_{\lambda \mu}(q,t) =  \langle m_{\gamma}^{(n)}(x) P_{\mu}^{(n)}(x;q,t),  Q_{\delta}^{(n)}(x;q,t) \rangle' \\= d_{\delta}(q,t)\langle m_{\gamma}^{(n)}(x) P_{\mu}^{(n)}(x;q,t),  P_{\delta}^{(n)}(x;q,t) \rangle 
 \end{multline*}
 and similarly
 \begin{equation*}
 \langle P_{\lambda}^{(n)}(x;t) s_{r}^{(n)}(x), Q_{\lambda+}^{(n)}(x;t) \rangle' = sk_{\lambda+/\lambda}(t). \end{equation*}

 \section{The Gelfand-Tsetlin basis expansion} \label{sec: findim}
 
 In this section, we fix $k \in \mathbb{N}$ and set $t = q^{k}$.  We will prove Theorem  \ref{qtthm} of the introduction.  Namely, we will expand the trace function $\Phi_{\lambda}^{(n)}(x)$ in the Gelfand-Tsetlin basis and compute the diagonal coefficients $c_{\Lambda}(q,t)$.  We will use the multiplicity-one decomposition of $V_{\lambda + (k-1)\rho}$ as a $U_{q}(gl_{n-1})$-module, and iterate, in order to do this.

Etingof and Kirillov \cite{EK} provide the following closed form for the trace function at $\lambda = 0$:
 
 \begin{proposition} \label{tracezero}
 \begin{equation*}
 \Phi_{0}^{(n)}(x) = \prod_{i=1}^{k-1} \prod_{\alpha \in R^{+}} (x^{\alpha/2} - q^{2i}x^{-\alpha/2}) = x^{(k-1)\rho} \prod_{i=1}^{k-1} \prod_{n \geq l > m \geq 1} (1-q^{2i}x_{l}/x_{m})
 \end{equation*}
 \end{proposition}
 
\begin{proposition} \label{tracezeroratio}
\begin{equation*}
\frac{\phi_{0}^{(n)}(x)}{\phi_{0}^{(n-1)}(x)} 
= x^{(k-1)(\rho_{n}-\rho_{n-1})} \sum_{l_{1}, \dots, l_{n-1}=0}^{\infty} t^{2\sum l_{i}} \frac{(t^{-2}q^{2};q^{2})_{l_{1}} \cdots (t^{-2}q^{2};q^{2})_{l_{n-1}}}{(q^{2};q^{2})_{l_{1}} \cdots (q^{2};q^{2})_{l_{n-1}}} \frac{x_{n}^{\sum l_{i}}}{x_{1}^{l_{1}} \cdots x_{n-1}^{l_{n-1}}}
\end{equation*}
\end{proposition}
 
\begin{proof}
By Proposition \ref{tracezero}, we have
\begin{equation*}
\frac{\phi_{0}^{(n)}(x)}{\phi_{0}^{(n-1)}(x)} = x^{(k-1)(\rho_{n}-\rho_{n-1})} \prod_{i=1}^{k-1}\prod_{j=1}^{n-1} (1-q^{2i}x_{n}/x_{j}).
\end{equation*}
Now note that, for fixed $1 \leq j \leq n-1$,
\begin{multline*}
\prod_{i=1}^{k-1} (1-q^{2i}x_{n}/x_{j}) = \frac{(x_{n}/x_{j};q^{2})_{k}}{(x_{n}/x_{j};q^{2})_{1}} = \frac{(x_{n}/x_{j};q^{2})_{\infty}}{(q^{2k}x_{n}/x_{j};q^{2})_{\infty}} \cdot \frac{(q^{2}x_{n}/x_{j};q^{2})_{\infty}}{(x_{n}/x_{j} ; q^{2})_{\infty}} \\ = \frac{(q^{2}x_{n}/x_{j};q^{2})_{\infty}}{(q^{2k}x_{n}/x_{j};q^{2})_{\infty}}.
\end{multline*}
Now, putting $t = q^{k}$ and using the $q$-binomial theorem,
\begin{equation*}
\frac{(q^{2}x_{n}/x_{j};q^{2})_{\infty}}{(t^{2}x_{n}/x_{j};q^{2})_{\infty}} = \sum_{m=0}^{\infty} \frac{(t^{-2}q^{2};q^{2})_{m}}{(q^{2};q^{2})_{m}}(t^{2}x_{n}/x_{j})^{m}.
\end{equation*}
Taking the product over all $1 \leq j \leq n-1$ and multiplying by $x^{(k-1)(\rho_{n}-\rho_{n-1})}$ gives the result.

\end{proof} 

\begin{lemma}
Let $\lambda$ be fixed with $l(\lambda) = n$.  Then the map 
\begin{equation*}
\mu \rightarrow \bar{\mu} = \mu + (k-1)\Big(\rho_{n-1} + \Big(\frac{1}{2}, \frac{1}{2}, \dots, \frac{1}{2}\Big) \Big) = \mu + (k-1)\rho_{n}|_{n-1}
\end{equation*}
is a bijection between:
\begin{itemize}
\item $\mu \subset \lambda$, such that $\lambda_{j+1} - \mu_{j} \leq k-1$ for all $j$
\item $\bar{\mu} \preceq \lambda + (k-1)\rho_{n}$, such that $\bar{\mu} - (k-1)\rho_{n}|_{n-1} \in \mathcal{P}_{+}$
\end{itemize}
\end{lemma}
\begin{proof}
Follows from the definition of the interlacing condition and Equation \ref{rhoeq}.
\end{proof}

 \begin{proposition}\label{finbrrule}
 The following branching rule for trace functions holds:
 \begin{equation*}
 \Phi_{\lambda}^{(n)}(x;q,q^{k}) = (x_{1} \cdots x_{n-1})^{\frac{k-1}{2}}\sum_{\mu \subset \lambda} x_{n}^{\rho(\lambda, \mu)} a_{\lambda, \mu}(q) \Phi_{\mu}^{(n-1)}(x;q,q^{k})
 \end{equation*}
 for some coefficients $a_{\lambda, \mu}(q)$.
 \end{proposition}
 \begin{proof}
 One first notes the multiplicity-free decomposition of $V_{\lambda + (k-1)\rho_{n}}$ as a module over $U_{q}(gl_{n-1})$:
 \begin{equation*}
 V_{\lambda + (k-1)\rho_{n}}|_{U_{q}(gl_{n-1})} = \oplus_{\bar{\mu} \preceq \lambda + (k-1)\rho_{n}} V_{\bar{\mu}}.
  \end{equation*}
 Thus, we have 
 \begin{multline*}
 \Phi_{\lambda}(x;q,q^{k}) = \Tr(\phi_{\lambda} \cdot x^{h}) = \sum_{\bar{\mu} \preceq \lambda + (k-1)\rho_{n} } \Tr(\phi_{\lambda} \cdot x^{h} |_{V_{\bar{\mu}}}) \\ = \sum_{\bar{\mu} \preceq \lambda + (k-1)\rho_{n} } \Tr\Big( (\Proj_{V_{\bar{\mu}}} \otimes Id) \circ \phi_{\lambda} \circ x^{h}|_{V_{\bar{\mu}}} \Big),
 \end{multline*}
 since trace only takes into account diagonal coefficients.  

We have
\begin{equation*}
\phi_{\lambda}|_{V_{\bar{\mu}}}: V_{\bar{\mu}} \rightarrow V_{\lambda + (k-1)\rho_{n}} \otimes U \simeq \oplus_{\alpha \preceq \lambda + (k-1)\rho_{n}} V_{\alpha} \otimes U,
\end{equation*}
thus
 \begin{equation*}
 (\Proj_{V_{\bar{\mu}}} \otimes Id) \circ \phi_{\lambda}|_{V_{\bar{\mu}}} : V_{\bar{\mu}} \rightarrow V_{\bar{\mu}} \otimes U
 \end{equation*}
 is an intertwining operator.  By \cite{EK}, this implies that
 \begin{equation*}
  (\Proj_{V_{\bar{\mu}}} \otimes Id) \circ \phi_{\lambda}|_{V_{\bar{\mu}}} = \begin{cases} \hat{a}_{\lambda, \bar{\mu}}(q) \cdot \phi_{\bar{\mu} - (k-1)\rho_{n-1}}, & \text{if } \bar{\mu} - (k-1)\rho_{n-1} \in \mathcal{P}_{+} \\
  0, & \text{else,}
  \end{cases}
 \end{equation*}
for some coefficients $\hat{a}_{\lambda, \bar{\mu}}(q)$.  Thus, we have,
\begin{equation*}
\Phi_{\lambda}(x;q,q^{k}) = \sum_{\substack{\bar{\mu} \preceq \lambda + (k-1)\rho_{n} \\ \bar{\mu} - (k-1)\rho_{n-1} \in \mathcal{P}_{+}}} \hat{a}_{\lambda, \bar{\mu}}(q) \cdot x_{n}^{\rho(\lambda, \bar{\mu})} \cdot \Phi_{\bar{\mu} - (k-1)\rho_{n-1}}(x;q,q^{k}).
\end{equation*}
Finally we reparametrize by setting $\mu = \bar{\mu} - (k-1)\rho_{n-1} - (k-1)(\frac{1}{2})^{n-1}$ and defining $a_{\lambda, \mu}(q) = \hat{a}_{\lambda, \bar{\mu}}(q)$ with the condition that $a_{\lambda,\mu} = 0$ if $\lambda_{j+1} - \mu_{j} \leq k-1$ does not hold for all $j$.  The result now follows by the previous Lemma.

 \end{proof}

By iterating the branching rule of the previous propostion and recalling that the Gelfand-Tsetlin basis is also obtained by iterating the multiplicty-free decomposition, we obtain the following result.
\begin{proposition}
We have the following formula for $\Phi_{\lambda}^{(n)}(x;q,q^{k})$ as a sum over Gelfand-Tsetlin patterns $(\lambda^{(0)}, \lambda^{(1)}, \dots, \lambda^{(n-1)})$ with $\lambda^{(0)} = \lambda + (k-1)\rho_{n}$:
 \begin{equation*}
 \Phi_{\lambda}^{(n)}(x;q,q^{k}) = \sum_{\substack{\Lambda \in GT(\lambda + (k-1)\rho_{n})\\ \Lambda = (\lambda^{(0)}, \cdots \lambda^{(n-1)})}}\prod_{1 \leq i \leq n-1}a_{\lambda^{(i-1)}, \lambda^{(i)}}(q) x^{\text{wt}(\Lambda)}
 \end{equation*}
 \end{proposition}
We will show that the coefficients $a_{\lambda,\mu}(q)$ are equal to $c_{\lambda, \mu}(q,q^{k})$ defined in the introduction.  We will prove this through a series of propositions.  Recall the definitions of the functions $g( \cdot; \cdot, \cdot), \psi_{\cdot/\cdot}(\cdot, \cdot), c_{\cdot,\cdot}^{\cdot}(\cdot;\cdot)$ in the introduction.

\begin{lemma}
  For any $m \in \C$, the branching coefficients $a_{\lambda,\mu}(q)$ satisfy the shift invariance:
  \begin{equation*}
    a_{\lambda + m^n, \mu + m^{n-1}}(q) = a_{\lambda, \mu}(q).
  \end{equation*}
\end{lemma}

\begin{proof}
  The intertwining operator $\phi^{(k)}$, as well as the multiplicity one
  decomposition used in the proof of Proposition \ref{finbrrule}, is determined by the
  $U_q(\mathfrak{sl_n})$-module structure.  Indeed, $U_q(\mathfrak{gl_n})$
  differs only from $U_q(\mathfrak{sl_n})$ by the addition of the central
  element $q^{\epsilon_1 + \dotsb + \epsilon_n}$.  The result then follows
  easily from the observation that, as $U_q(\mathfrak{sl_n})$-modules,
  $V_{\lambda + m^n}$ is isomorphic to $V_{\lambda}$ for any partition
  $\lambda$ and any $m \in \C$.

\end{proof}
 
\begin{remark}
By the previous Lemma, to compute $a_{\lambda,\mu}(q)$ for $\lambda \in \mathcal P_+^{(n)}$, $\mu \in \mathcal P_+^{(n-1)}$, we may assume that $\lambda, \mu$ are partitions with $l(\lambda) = n$, $l(\mu) = n-1$.  We will make this assumption implicitly throughout the paper.
\end{remark}

 \begin{proposition} The branching coefficients satisfy the following formula:
\begin{equation*}
a_{\lambda, \mu}(q) = \sum_{\substack{\beta \preceq \lambda, l(\beta) \leq n-1 \\ \gamma \in \mathbb{Z}_{\geq 0}^{n-1} \\ \text{a partition}}} g(\gamma; q^{2}, q^{2k}) \psi_{\lambda/\beta}(q^{2}, q^{2k}) c_{-\gamma + (k-1)^{n-1}, \beta}^{\mu + (k-1)^{n-1}}(q^{2},q^{2k}).
\end{equation*}
\end{proposition}
 
 \begin{proof}
 Combining Theorem \ref{EKfin} with Proposition \ref{finbrrule} gives the following:
 \begin{equation*}
 \frac{\phi_{0}^{(n)}(x;q,q^{k})}{\phi_{0}^{(n-1)}(x;q,q^{k})} P_{\lambda}^{(n)}(x;q^{2},q^{2k}) = \sum_{\mu \subset \lambda} x_{n}^{\rho(\lambda, \mu)} a_{\lambda, \mu}(q) P_{\mu + (k-1)(\frac{1}{2})^{n-1}}^{(n-1)}(x;q^{2},q^{2k}).
 \end{equation*}
 We then use Theorem \ref{MDbr} to rewrite this as
 \begin{multline} \label{starteqn}
  \frac{\phi_{0}^{(n)}(x;q,q^{k})}{\phi_{0}^{(n-1)}(x;q,q^{k})} \sum_{\substack{\mu \preceq \lambda \\ l(\mu) \leq n-1}} x_{n}^{|\lambda - \mu|} \psi_{\lambda/\mu}(q^{2},q^{2k}) P_{\mu}^{(n-1)}(x;q^{2},q^{2k})\\ = \sum_{\mu \subset \lambda} x_{n}^{\rho(\lambda, \mu)} a_{\lambda, \mu}(q) P_{\mu + (k-1)(\frac{1}{2})^{n-1}}^{(n-1)}(x;q^{2},q^{2k}).
 \end{multline}
 Now note from Proposition \ref{tracezeroratio}, we have
 \begin{multline*}
 \frac{\phi_{0}^{(n)}(x;q,q^{k})}{\phi_{0}^{(n-1)}(x;q,q^{k})} 
 = x^{(k-1)(\rho_{n}-\rho_{n-1})} \times \\
\times \sum_{\gamma \in \mathbb{Z}_{\geq 0}^{n-1}} q^{2k(\sum_{i} \gamma_{i})} \frac{(q^{-2(k-1)};q^{2})_{\gamma_{1}} \cdots (q^{-2(k-1)};q^{2})_{\gamma_{n-1}}}{(q^{2};q^{2})_{\gamma_{1}} \cdots (q^{2};q^{2})_{\gamma_{n-1}}} x_{n}^{(\sum_{i} \gamma_{i})} m_{-\gamma}(x_{1}, \dots, x_{n-1}),
 \end{multline*}
where the sum is over $\gamma = (\gamma_{1}, \dots, \gamma_{n-1})$ a partition.  Now recall that $(k-1)(\rho_{n} - \rho_{n-1}) = (k-1)(\frac{1}{2}, \frac{1}{2}, \dots, \frac{1}{2}, \frac{1-n}{2}) \in \mathbb{Z}^{n}$, so we have
\begin{equation*}
\frac{\phi_{0}^{(n)}(x;q,q^{k})}{\phi_{0}^{(n-1)}(x;q,q^{k})} = \sum_{\substack{\gamma \in \mathbb{Z}_{\geq 0}^{n-1}\\ \text{a partition}}}  x_{n}^{(k-1)(\frac{1-n}{2}) + |\gamma|} g(\gamma;q^{2},q^{2k}) m_{-\gamma + (k-1)(\frac{1}{2}, \dots, \frac{1}{2})}(x_{1}, \dots, x_{n-1}),
\end{equation*}
where we have used Definition (\ref{gfcn}).  

 We use the previous equation, along with (\ref{starteqn}), and multiply both sides by the monomial $(x_{1} \cdots x_{n-1})^{(k-1)(\frac{1}{2}, \dots, \frac{1}{2})}$ to obtain the equation:

  \begin{multline*}
 \sum_{\substack{\gamma \in \mathbb{Z}_{\geq 0}^{n-1} \\ \text{a partition} \\ \mu \preceq \lambda \\ l(\mu) \leq n-1}}  x_{n}^{|\lambda - \mu|+ (k-1)(\frac{1-n}{2}) + |\gamma|} g(\gamma;q^{2},q^{2k}) \psi_{\lambda/\mu}(q^{2},q^{2k})\times \\ \times m_{-\gamma + (k-1)^{n-1}}(x_{1}, \dots, x_{n-1})    P_{\mu}^{(n-1)}(x;q^{2},q^{2k}) \\= \sum_{\mu \subset \lambda} x_{n}^{\rho(\lambda, \mu)} a_{\lambda,\mu}(q) P_{\mu+ (k-1)^{n-1}}^{(n-1)}(x;q^{2},q^{2k}).
 \end{multline*}
Next we use Definition (\ref{mPieri}) to rewrite this as
\begin{multline*}
 \sum_{\substack{\gamma \in \mathbb{Z}_{\geq 0}^{n-1} \\ \text{a partition} \\ \mu \preceq \lambda, l(\mu) \leq n-1 \\ \delta \text{ a partition}}}  x_{n}^{|\lambda - \mu|+ (k-1)(\frac{1-n}{2}) + |\gamma|} g(\gamma;q^{2},q^{2k}) \psi_{\lambda/\mu}(q^{2},q^{2k}) \times \\ \times c_{-\gamma + (k-1)^{n-1}, \mu}^{\delta}(q^{2}, q^{2k}) P_{\delta}^{(n-1)}(x;q^{2},q^{2k}) \\
 = \sum_{\mu \subset \lambda} x_{n}^{\rho(\lambda, \mu)} a_{\lambda, \mu}(q) P_{\mu+ (k-1)^{n-1}}^{(n-1)}(x;q^{2},q^{2k}).
\end{multline*}
Since both LHS and RHS are expansions in the Macdonald polynomial basis, the corresponding coefficients must be equal.  That is,
\begin{equation*}
a_{\lambda, \mu}(q) = \sum_{\substack{\beta \preceq \lambda, l(\beta) \leq n-1 \\ \gamma \in \mathbb{Z}_{\geq 0}^{n-1} \\ \text{a partition}}} g(\gamma; q^{2}, q^{2k}) \psi_{\lambda/\beta}(q^{2}, q^{2k}) c_{-\gamma + (k-1)^{n-1}, \beta}^{\mu + (k-1)^{n-1}}(q^{2},q^{2k}),
\end{equation*}
as desired.
 \end{proof}

\begin{proposition}
Let $\mu \subset \lambda$ with $l(\mu) \leq n-1$.  Then we have
\begin{equation*}
a_{\lambda, \mu}(q) = \sum_{\substack{\beta \preceq \lambda, l(\beta) \leq n-1 \\ l(\gamma) \leq n-1}} g(\gamma; q^{2}, q^{2k}) \psi_{\lambda/\beta}(q^{2}, q^{2k})  \frac{d_{\mu}(q^{2},q^{2k})}{d_{\beta}(q^{2},q^{2k})} c_{\gamma, \mu}^{\beta}(q^{2},q^{2k}).
\end{equation*}
\end{proposition} 
 \begin{proof}
Using standard facts about integration over $\mathbb{T}_{n}$, we have
\begin{multline*}
c_{-\gamma + (k-1)^{n-1}, \beta}^{\mu + (k-1)^{n-1}}(q^{2},q^{2k})  \\ = \langle m_{-\gamma + (k-1)^{n-1}}(x) P_{\beta}(x; q^{2},q^{2k}), P_{\mu + (k-1)^{n-1}}(x;q^{2},q^{2k}) \rangle d_{\mu + (k-1)^{n-1}}(q^{2},q^{2k}) \\ 
= \langle m_{-\gamma}(x) P_{\beta}(x; q^{2},q^{2k}), P_{\mu}(x;q^{2},q^{2k}) \rangle d_{\mu + (k-1)^{n-1}}(q^{2},q^{2k}) \\
= d_{\mu + (k-1)^{n-1}}(q^{2},q^{2k}) \langle m_{\gamma}(x) P_{\mu}(x;q^{2},q^{2k}) ,  P_{\beta}(x; q^{2},q^{2k})\rangle \\
= \frac{d_{\mu}(q^{2},q^{2k}) }{d_{\beta}(q^{2},q^{2k})} c_{\gamma, \mu}^{\beta}(q^{2}, q^{2k}),
\end{multline*}
where we have used $d_{\mu + (k-1)^{n-1}}(q^{2}, q^{2k}) = d_{\mu}(q^{2}, q^{2k})$, which follows from the definition of $\langle \cdot, \cdot \rangle$. Combining this with the previous theorem gives the result.
 \end{proof}

 We are now prepared to provide a proof of Theorem \ref{qtthm}, mentioned in the introduction to this paper.  The proof relies on the previous propositions proved in this section.

 \begin{proof}[Proof of Theorem \ref{qtthm}]
 By the previous proposition, we have
\begin{equation*}
a_{\lambda, \mu}(q) = d_{\mu}(q^{2},q^{2k}) \sum_{\substack{\beta \preceq \lambda, l(\beta) \leq n-1 \\ l(\gamma) \leq n-1}} g(\gamma; q^{2}, q^{2k}) \psi_{\lambda/\beta}(q^{2}, q^{2k}) \frac{1}{d_{\beta}(q^{2},q^{2k})} c^{\beta}_{\gamma, \mu}(q^{2}, q^{2k}).
\end{equation*}
Now, note that for fixed $\beta$, we have
\begin{multline*}
\sum_{l(\gamma) \leq n-1} g(\gamma; q^{2},q^{2k}) c^{\beta}_{\gamma, \mu}(q^{2}, q^{2k}) \\= \sum_{l(\gamma) \leq n-1}q^{2k|\gamma|} \frac{(q^{-2k}q^{2};q^{2})_{\gamma_{1}} \cdots (q^{-2k}q^{2};q^{2})_{\gamma_{n-1}}}{(q^{2};q^{2})_{\gamma_{1}} \cdots (q^{2};q^{2})_{\gamma_{n-1}}} c_{\gamma, \mu}^{\beta}(q^{2}, q^{2k}) \\
= \sum_{l(\gamma) \leq n-1}q^{2k|\gamma|} \frac{(q^{-2k}q^{2};q^{2})_{\gamma_{1}} \cdots (q^{-2k}q^{2};q^{2})_{\gamma_{n-1}}}{(q^{2};q^{2})_{\gamma_{1}} \cdots (q^{2};q^{2})_{\gamma_{n-1}}} \langle m_{\gamma}^{(n-1)}(x) P_{\mu}^{(n-1)}(x;q^{2}, q^{2k}), Q_{\beta}^{(n-1)}(x;q^{2}, q^{2k}) \rangle' \\
= \Big\langle \Big( \sum_{l(\gamma) \leq n-1}q^{2k|\gamma|} \frac{(q^{-2k}q^{2};q^{2})_{\gamma_{1}} \cdots (q^{-2k}q^{2};q^{2})_{\gamma_{n-1}}}{(q^{2};q^{2})_{\gamma_{1}} \cdots (q^{2};q^{2})_{\gamma_{n-1}}} m_{\gamma}^{(n-1)}(x) \Big)P_{\mu}^{(n-1)}(x;q^{2}, q^{2k}), Q_{\beta}^{(n-1)}(x;q^{2}, q^{2k}) \Big\rangle'.
\end{multline*} 
Now recall the following identity \cite[p314]{Mac}: 
\begin{equation*}
g_{n}(x;q^{2}, t^{2}) = \sum_{|\mu| = n} \frac{(t^{2};q^{2})_{\mu}}{(q^{2};q^{2})_{\mu}} m_{\mu}(x).
\end{equation*}
Using this, we may write the previous equation as
\begin{multline*}
\sum_{l(\gamma) \leq n-1} g(\gamma; q^{2},q^{2k}) c^{\beta}_{\gamma, \bar{\mu}}(q^{2}, q^{2k}) \\= \Big\langle \Big( \sum_{r \geq 0}q^{2kr} g_{r}^{(n-1)}(x;q^{2}, q^{-2k}q^{2}) \Big)P_{\mu}^{(n-1)}(x;q^{2}, q^{2k}), Q_{\beta}^{(n-1)}(x;q^{2}, q^{2k}) \Big\rangle'.
\end{multline*} 
Also note \cite[p311]{Mac} that we have the generating function identity (for arbitrary $q,t,y$):
 \begin{equation*}
 \sum_{n \geq 0} g_{n}(x;q,t) y^{n} = \prod_{i \geq 1} \frac{(tx_{i}y;q)_{\infty}}{(x_{i}y;q)_{\infty}};
  \end{equation*}
thus we have
 \begin{multline*}
\sum_{l(\gamma) \leq n-1} g(\gamma; q^{2},q^{2k}) c^{\beta}_{\gamma, \mu}(q^{2}, q^{2k}) \\= \Big\langle \prod_{i \geq 1} \frac{(q^{2}x_{i};q^{2})_{\infty}}{(q^{2k}x_{i};q^{2})_{\infty}} P_{\mu}^{(n-1)}(x;q^{2}, q^{2k}), Q_{\beta}^{(n-1)}(x;q^{2}, q^{2k}) \Big\rangle'.
\end{multline*} 
Combining this with the original sum yields
 \begin{multline*}
a_{\lambda, \mu}(q)  = d_{\mu}(q^{2},q^{2k}) \times \\ \times \sum_{\substack{\beta \preceq \lambda \\l(\beta) \leq n-1 }}  \frac{\psi_{\lambda/\beta}(q^{2}, q^{2k})}{d_{\beta}(q^{2}, q^{2k})} \Big\langle \prod_{i \geq 1} \frac{(q^{2}x_{i};q^{2})_{\infty}}{(q^{2k}x_{i};q^{2})_{\infty}} P_{\mu}^{(n-1)}(x;q^{2}, q^{2k}), Q_{\beta}^{(n-1)}(x;q^{2}, q^{2k}) \Big\rangle' \\
= d_{\mu}(q^{2},q^{2k}) \sum_{\substack{\beta \preceq \lambda \\ l(\beta) \leq n-1}}  \frac{\psi_{\lambda/\beta}(q^{2}, q^{2k})}{d_{\beta}(q^{2}, q^{2k})} \Omega_{\beta/ \mu}(q^{2}, q^{2k})
\end{multline*}
 by definition of $\Omega_{\beta/\mu}(q^{2}, q^{2k})$.  But this is exactly equal to $c_{\lambda,\mu}(q,q^{k})$ as defined in (\eqref{eq:branchcoeff}).
\end{proof}

\begin{remarks}
The coefficients $c_{\lambda,\mu}(q,t)$ do not appear to factor nicely at the $q$-level, due to the restriction on length in the sum.  For example, for $\lambda=(2,1)$ and
$\mu=(1)$ one obtains $(1-t)(1-qt^2)(1-q-q^2+t)/(1-qt)$, and the term $(1-q-q^{2}+t)$ cannot be expressed as a product of $(1-q^{i}t^{j})$.

\end{remarks}

\section{The $q \to 0$ limit}  
We will look at the $q \rightarrow 0$ limit of the coefficients $c_{\lambda,\mu}(q,t)$.  We find that the formula has a nice product form, in terms of certain $p$-adic counts.  The simplification of these coefficients at $q=0$ may be related to the crystal basis structure of the Gelfand-Tsetlin basis, although we have not investigated a direct link.

We note that the parameters $q, t$ are linked in the finite-dimensional case since $t = q^{k}$ and $k \in \mathbb{N}$, so we cannot take $q \rightarrow 0$ without having $t \rightarrow 0$ as well.  In this section, we will simply treat $t$ as a formal variable and in Section 5, we will relate it to the representation theory by using Verma modules.  
 
 The goal of this section is to prove Theorems \ref{trestr} and \ref{thm:gt_padic} mentioned in the introduction.
 
 \begin{definition}
We let $c_{\lambda, \mu}(t)$ denote $\displaystyle \lim_{q \rightarrow 0} c_{\lambda, \mu}(q,t)$.  
\end{definition}

 \begin{theorem}
 Let $\lambda$ be a partition of length $n$, and $\mu \subset \lambda, \mu \in \mathcal{P}_{+}$.  Then retaining the notation of the previous sections, we have
 \begin{equation*}
 c_{\lambda, \mu}(t) = \frac{b_{\mu}(t^{2})}{b_{\lambda}(t^{2})} \sum_{\substack{\beta \preceq \lambda \\ l(\beta) \leq n-1}} \phi_{\lambda/ \beta}(t^{2}) t^{2|\beta/\mu|} sk_{\beta/\mu}(t^{2}).
 \end{equation*}
 \end{theorem}
 
 \begin{proof}
 We use Theorem \ref{qtthm}, the functions there admit the limit $q \rightarrow 0$.
 We also use that
 \begin{equation*}
\Omega_{\beta / \mu}(q,t)= Q_{\beta / \mu} \Big( \frac{t-q}{1-t}\Big) = t^{|\beta/\mu|} Q_{\beta/\mu} \Big( \frac{1-q}{1-t} \Big)
 \end{equation*}
 and
 \begin{equation*}
 \lim_{q \rightarrow 0} Q_{\beta/\mu} \Big( \frac{1-q}{1-t} \Big) = sk_{\beta/\mu}(t),
 \end{equation*}
 where the skew Macdonald polynomials are taken with respect to the parameters $(q,t)$.  This gives the following
  \begin{equation*}
 c_{\lambda, \mu}(t) = d_{\mu}(t^{2}) \sum_{\substack{\beta \preceq \lambda \\ l(\beta) \leq n-1}} \frac{\psi_{\lambda/ \beta}(t^{2})}{d_{\beta}(t^{2})} t^{2|\beta/\mu|} sk_{\beta/\mu}(t^{2}).
 \end{equation*}
 Finally, one notes that
 \begin{equation*}
 d_{\beta}(t^{2}) = b_{\beta}(t^{2}) \text{ and } d_{\mu}(t^{2}) = b_{\mu}(t^{2}),
  \end{equation*}
 and by the $q \rightarrow 0$ limit of Proposition \ref{phipsi} we have
\begin{equation*}
\phi_{\lambda/\beta}(t)/b_{\lambda}(t) = \psi_{\lambda/\beta}(t)/b_{\beta}(t);
\end{equation*}
using this in the previous equation gives the result.

 \end{proof}
 
Our next goal is to obtain a nice factorized product form for $c_{\lambda,
\mu}(t)$.  We use a $q\to 0$ specialization of Rains' $q$-Pfaff-Saalsch\"{u}tz
formula: 

\begin{theorem}[{\cite[Corollary 4.9]{R}}] \label{q_pfaff}
Let $\mu \subset \lambda$ be partitions, then for arbitrary parameters $a, b,
c$ we have the following identity:
  \begin{equation*}
    \sum_{\beta} \frac{(a)_\beta}{(c)_\beta} Q_{\lambda /
    \beta}\biggl(\frac{a-b}{1-t}\biggr)Q_{\beta/\mu}\biggl(\frac{b-c}{1-t}\biggr)
    = \frac{(a)_\mu (b)_{\lambda}}{(b)_\mu (c)_\lambda}Q_{\lambda/\mu}\biggl(\frac{a-c}{1-t}\biggr).  
  \end{equation*}
\end{theorem}

\begin{proposition}\label{hl_pfaff_saal}
Let $\lambda$ be a partition of length $n$ and let $\mu \leq \lambda$ with $l(\mu) = n-1$.  Then 
\begin{enumerate}
  \item 
    \begin{equation*}
      \sum_{\beta \preceq \lambda} \phi_{\lambda/ \beta}(t^{2}) t^{2|\beta/\mu|} sk_{\beta/\mu}(t^{2}) = sk_{\lambda/\mu}(t^2)
    \end{equation*}
  \item 
    \begin{equation*}
      \sum_{\substack{\beta \preceq \lambda \\ l(\beta) \leq n-1}} \phi_{\lambda/ \beta}(t^{2}) t^{2|\beta/\mu|} sk_{\beta/\mu}(t^{2}) = (1-t^2) \cdot sk_{\lambda/\mu}(t^2)
    \end{equation*}
\end{enumerate}
\end{proposition}

\begin{proof}
Take $a = q$, $b = qt$, $c = q^2$ in in Theorem \ref{q_pfaff}:
\begin{equation*}
\sum_{\beta} \frac{(q)_\beta}{(q^2)_\beta} Q_{\lambda / \beta}\biggl(\frac{q-qt}{1-t}\biggr)Q_{\beta/\mu}\biggl(\frac{qt-q^2}{1-t}\biggr)
  = \frac{(q)_\mu (qt)_{\lambda}}{(qt)_\mu (q^2)_\lambda}Q_{\lambda/\mu}\biggl(\frac{q-q^2}{1-t}\biggr).  
\end{equation*}
Using the relation $Q_{\lambda / \mu}\bigl(\frac{aq-bq}{1-t}\bigr) = q^{|\lambda / \mu|} Q_{\lambda / \mu}\bigl(\frac{a-b}{1-t}\bigr)$, we have
\begin{equation*}
\sum_{\beta} \frac{(q)_\beta}{(q^2)_\beta} Q_{\lambda / \beta}\bigl(1\bigr)Q_{\beta/\mu}\biggl(\frac{t-q}{1-t}\biggr)
  = \frac{(q)_\mu (qt)_{\lambda}}{(qt)_\mu (q^2)_\lambda}Q_{\lambda/\mu}\biggl(\frac{1-q}{1-t}\biggr).  
\end{equation*}
(1) then follows by taking the limit $q \to 0$.

To prove (2), it suffices by (1) to show that
\begin{equation*}
  \sum_{\substack{\beta \preceq \lambda \\ l(\beta) = n}} \phi_{\lambda/ \beta}(t^{2}) t^{2|\beta/\mu|} sk_{\beta/\mu}(t^{2}) = t^2 \cdot sk_{\lambda/\mu}(t^2).
\end{equation*}
Reindex the sum by replacing $\beta = \beta' + 1^n$.  Then we have 
\begin{equation*}
  \sum_{\beta' \preceq \lambda - 1^n} \phi_{\lambda/ (\beta' + 1^n)}(t^{2}) t^{2|\beta'/\mu|} t^{2n} sk_{(\beta' + 1^n)/\mu}(t^{2}) 
\end{equation*}
We have the following two identities:
\begin{align*}
  \phi_{\lambda / (\beta' + 1^n)}(t^2) &= \phi_{(\lambda - 1^n) / \beta'}(t^2) \\
  sk_{(\beta' + 1^n) / \mu}(t^2)       &= sk_{\beta' / (\mu - 1^{(n-1)})}(t^2),
\end{align*}
which can be seen by using the explicit formulas in Section 2. It follows that
\begin{equation*}
  \sum_{\beta' \preceq (\lambda-1^n)} \phi_{(\lambda-1^n)/ \beta'}(t^{2}) t^{2|\beta'/(\mu-1^{n-1})|} t^2 sk_{\beta'/(\mu-1^{(n-1)})}(t^{2}) = t^2 \cdot sk_{(\lambda - 1^n)/(\mu -1^{(n-1)})}(t^2).
\end{equation*}
Applying the identity above again completes the proof.
\end{proof}

We now provide a proof of Theorem \ref{trestr}, mentioned in the introduction.  The proof relies on the previous results of this section.

 \begin{proof}[Proof of Theorem \ref{trestr}]
 By Proposition \ref{hl_pfaff_saal}, we have
 \begin{equation*}
 c_{\lambda, \mu}(t) = \frac{b_{\mu}(t^{2})}{b_{\lambda}(t^{2})} \sum_{\substack{\beta \preceq \lambda \\ l(\beta) \leq n-1}} \phi_{\lambda/ \beta}(t^{2}) t^{2|\beta/\mu|} sk_{\beta/\mu}(t^{2}) = \frac{b_{\mu}(t^{2})}{b_{\lambda}(t^{2})} (1-t^{2}) sk_{\lambda/\mu}(t^{2}),
 \end{equation*}
 which gives the first equality.
 By the definitions of $b_{\lambda}(t), sk_{\lambda/\mu}(t)$, this is equal to
 \begin{multline*}
 \frac{\prod_{i \geq 1} \phi_{m_{i}(\mu)}(t^{2})}{\prod_{i \geq 1} \phi_{m_{i}(\lambda)}(t^{2})}(1-t^{2}) t^{2\sum_{j} \binom{\lambda_{j}'-\mu_{j}'}{2}} \prod_{j \geq 1} \binom{\lambda_{j}'-\mu_{j+1}'}{\mu_{j}' - \mu_{j+1}'}_{t^{2}} \\
 = (1-t^{2}) t^{2\sum_{j} \binom{\lambda_{j}'-\mu_{j}'}{2}}\prod_{j \geq 1} \frac{\phi_{\lambda_{j}'-\mu_{j+1}'}(t^{2})}{\phi_{\lambda_{j}' - \lambda_{j+1}'}(t^{2}) \phi_{\lambda_{j}' - \mu_{j}'}(t^{2})} 
 \\= \frac{(1-t^{2})}{\phi_{\lambda_{1}'-\mu_{1}'}(t^{2})} t^{2\sum_{j}  \binom{\lambda_{j}'-\mu_{j}'}{2}} \prod_{j \geq 1} \frac{\phi_{\lambda_{j}'-\mu_{j+1}'}(t^{2})}{\phi_{\lambda_{j}' - \lambda_{j+1}'}(t^{2}) \phi_{\lambda_{j+1}' - \mu_{j+1}'}(t^{2})}\\
 = t^{2\sum_{j} \binom{\lambda_{j}'-\mu_{j}'}{2}} \prod_{j \geq 1} \binom{\lambda_{j}'-\mu_{j+1}'}{\lambda_{j}'-\lambda_{j+1}'}_{t^{2}},
 \end{multline*}
where we have used $m_{i}(\mu) = \mu_{i}' - \mu_{i+1}'$ and $\lambda_{1}' - 1 =
\mu_{1}'$ (because $l(\lambda) = n$ and $l(\mu) = n-1$).
\end{proof}

 We recall that, as mentioned in the introduction, there is a $p$-adic interpretation for coefficients $sk_{\lambda/\mu}(t)$ and thus for $c_{\lambda, \mu}(t)$.  More precisely,
 \begin{equation*}
 sk_{\lambda/\mu}(t) = t^{n(\lambda) - n(\mu)} \alpha_{\lambda}(\mu; t^{-1});
 \end{equation*}
 where $\alpha_{\lambda}(\mu;p)$ is the number of subgroups of type $\mu$ in a finite abelian $p$-group of type $\lambda$, see \cite{W} for example, and the references therein.

\begin{proof}[Proof of Theorem \ref{thm:gt_padic}]
Recall that for $S = (\mu^{(0)} \supset \mu^{(1)} \supset \dotsb \supset \mu^{(n-1)})$ with $\mu^{(i)} \in \mathcal P_+^{(n-i)}$, we defined the coefficient $sk_{S}(t)$ as a product of $sk_{\mu^{(i-1)}/\mu^{(i)}}(t)$ in (\ref{skScoeff}).  By Definition \ref{def:gt_index}, one can associate to $S$ a Gelfand-Tsetlin array $\Lambda$.  Thus, using Theorem \ref{trestr}, we have
\begin{equation*}
\lim_{q \rightarrow 0} c_{\Lambda}(q,t) = \frac{(1-t^{2})^{n}}{b_{\lambda}(t^{2})}sk_{S}(t^{2}).
\end{equation*}
Using this along with Theorem \ref{qtthm} gives the result.
\end{proof}
 
Note that when $t=p^{-1}$ for $p$ an odd prime, the coefficients appearing in both Theorems \ref{trestr} and \ref{thm:gt_padic} are explicit $p$-adic counts.

 \begin{corollary}
 Let $\lambda$ be a partition.  We have the following formula for the Hall-Littlewood polynomial:
 \begin{equation*} 
 P_{\lambda}(x_{1}, \dots, x_{n}; t^{2}) = \frac{1}{b_{\lambda}(t^{2})} \frac{\displaystyle\sum_{\substack{S = (\lambda = \mu^{(0)} \supset
    \mu^{(1)} \supset \dotsb \supset \mu^{(n-1)}) \\ \mu^{(i)} \in \mathcal P_+^{(n-i)}}} sk_{S}(t^{2})x^{wt(S)}}{\displaystyle \sum_{\substack{S' = ( 0^{n} = \mu^{(0)} \supset
    \mu^{(1)} \supset \dotsb \supset \mu^{(n-1)}) \\ \mu^{(i)} \in \mathcal P_+^{(n-i)}}} sk_{S'}(t^{2})x^{wt(S')}}.
 \end{equation*}
 \end{corollary}
 \begin{proof}
 Follows from Theorem \ref{EKfin} along with Theorem \ref{thm:gt_padic}.
 \end{proof}

\section{Verma modules and algebraically independent $t$} \label{sec:verma}

We have computed the expansion of the Macdonald vector-valued characters
$\Phi^{(k)}_\lambda(x; q)$ with respect to the Gelfand-Tsetlin basis of
$V_{\lambda + (k-1)\rho}$.  These are expressed in terms of rational functions
in $q,t$ which appear naturally in symmetric function theory, specialized to
$t=q^k$.  In the previous section we showed that, for algebraically independent
$t$, these coefficients admit a simple limit as $q \to 0$, which is related to
natural quantities appearing in $p$-adic representation theory.  Note however
that in the representation theoretic realization of $\Phi^{(k)}_\lambda$ we
have $t=q^k$, and hence we can only obtain the $t=0$ specialization of our
formula in the $q \to 0$ limit.

In \cite{EK}, Etingof and Kirillov showed that one can extend
$\Phi^{(k)}_\lambda$ to algebraically independent $t$ by replacing the
finite-dimensional irreducible module $V_{\lambda + (k-1)\rho}$ by a suitable
infinite dimensional irreducible Verma module.  In this section we outline
their construction, which allows us to obtain a representation theoretic
realization of our formula for algebraically independent $t$.

Consider the algebra $\C(t) \otimes U_q(\mathfrak{gl}_n)$, i.e. the quantum
group $U_q(\mathfrak {gl_n})$ where the coefficient field is expanded to $\C(t)
\otimes \C(q)$ (note this can be identified with the subalgebra of $\C(q,t)$
spanned by products of the form $r_1(q)\cdot r_2(t)$ for rational functions
$r_1, r_2$).  We have the following analogues of the finite-dimensional modules
$V_{\lambda + (k-1) \rho}$:

\begin{definition}
For $\lambda \in \mathcal P^{(n)}$, the module $M_{\lambda, t}$ over $\C(t) \otimes
U_q(\mathfrak{gl}_n)$ is uniquely defined by the following conditions:
\begin{enumerate}
  \item There is a \textit{highest weight vector} $m_\lambda \in M_{\lambda, t}$ satisfying:
    \begin{align*}
      e_i \cdot m_\lambda &= 0, & (1 \leq &i \leq n-1) \\
      q^{\epsilon_i} \cdot m_\lambda &= t^{2\rho_i} \cdot q^{2\cdot (\lambda_i-\rho_i)} \cdot m_\lambda, & (1 \leq &i \leq n)
    \end{align*}
  \item $g \mapsto g \cdot m_\lambda$ is a bijection from $\C(t) \otimes
    U_q^-(\mathfrak{gl}_n)$ to $M_{\lambda, t}$, where $U_q^{-}(\mathfrak{gl}_n)$ denotes the
    subalgebra generated by $f_1,\dotsc,f_{n-1}$.
\end{enumerate}
\end{definition}

Under the identification $t = q^k$, where $k$ is now a formal parameter, $M_{\lambda, t}$ is
isomorphic to the \textit{Verma module} of weight $\lambda + (k-1)\rho$ (see
e.g. \cite{EK}).  Let us attempt to clarify the relationship between
$M_{\lambda, t}$ and the finite-dimensional modules $V_{\lambda + (k-1)\rho}$.

Firstly, for $k \in \N$ there is a quotient mapping $\C(t) \otimes U_q(\mathfrak{gl}_n) \to
U_q(\mathfrak{gl}_n)$ sending $t\to q^k$.  Moreover, for $k\in \N$ then there is a $\C(q)$-linear
map $\alpha_k: M_{\lambda, t} \to V_{\lambda + (k-1)\rho}$ which is compatible with the module
structures in the sense that the following diagram commutes:
\begin{equation*}
  \xymatrix{
    \bigl(\C(t) \otimes U_q(\mathfrak{gl}_n)\bigr) \otimes M_{\lambda, t} \ar[d]_{(t
    \mapsto q^k) \otimes \alpha_k} \ar[rr]^(.65){g\otimes v \mapsto
    g\cdot v} &  & M_{\lambda, t}
      \ar[d]^{\alpha_k} \\
      U_q(\mathfrak{gl}_n) \otimes V_{\lambda + (k-1)\rho} \ar[rr]^(.55){g \otimes v \mapsto g\cdot v} &
      & V_{\lambda + (k-1)\rho}
}
\end{equation*}
(We take the convention here that $V_{\lambda + (k-1)\rho} = \{0\}$ if
$\lambda + (k-1)\rho \notin \mathcal P_+$, which can occur for only finitely
many $k$).  The kernels of $\alpha_k$ form a decreasing sequence of subspaces
of $M_{\lambda, t}$, and
\begin{equation*}
  \bigcap_{k \geq 0} \ker \alpha_k = \{0\}.
\end{equation*}
The existence of $\alpha_k$ satisfying the above conditions determines the module $M_{\lambda, t}$
uniquely.

The analogue of the finite-dimensional module $U \simeq V_{(k-1)\cdot (n-1,1,\dotsc,1)}$ is as
follows:
\begin{definition}
The module $W_{t}$ over $\C(t) \otimes U_q(\mathfrak{gl}_n)$ is the degree zero subspace of Laurent
polynomials
\begin{equation*}
  W_t = \{p(x) \in \C(t)\otimes \C(q)[x_1^{\pm 1} x_2^{\pm 1}\dotsb x_n^{\pm 1}]\; | \; \deg p = 0 \},
\end{equation*}
with the following action of the generators of $U_q(\mathfrak{gl}_n)$:
\begin{align*}
  \epsilon_i(p(x)) &= x_i \cdot \frac{\partial}{\partial x_i} p(x) \\
  e_i(p(x)) &= \frac{x_i}{x_{i+1}} \cdot \frac{(tq^{-1}) p(x_1,\dotsc,qx_{i+1},
  \dotsc, x_n)\negthinspace -\negthinspace (t^{-1}q) p(x_1,\dotsc,q^{-1}x_{i+1},\dotsc,x_n)}{(q-q^{-1})}\\
  f_i(p(x)) &= \frac{x_{i+1}}{x_{i}} \cdot \frac{(tq^{-1}) p(x_1,\dotsc,qx_{i},
  \dotsc, x_n) - (t^{-1}q) p(x_1,\dotsc,q^{-1}x_{i},\dotsc,x_n)}{(q-q^{-1})}
\end{align*}

\end{definition}

This is isomorphic to the module denoted $W_k$ in \cite{EK}, and is irreducible over $\C(t)
\otimes U_q(\mathfrak{gl}_n)$.  If $k \in \N$ is a fixed integer, we can quotient $W_t$ by
the relation $t = q^k$ to obtain an infinite-dimensional module over $U_q(\mathfrak{gl}_n)$.  This
is no longer irreducible, and the subspace spanned by $p(x)$ with $(x_1\dotsc x_n)^{(k-1)}\cdot p(x)
\in \C(q)[x_1,\dotsc,x_n]$ is identified with the module $U$.

It is shown in \cite{EK} that there is a unique intertwining operator
$\widetilde{\phi}: M_{\lambda, t} \to M_{\lambda,t} \otimes W_t$ if and only if
$\lambda \in \mathcal P_+$.  Moreover, for $k \in \N$ the intertwining
operators $\widetilde{\phi}$ and $\phi^{(k)}$ are compatible in the sense that
the following diagram commutes:
\begin{equation}\label{eq:verma_intertwine_compat}
  \xymatrix{M_{\lambda, t} \ar[r]^(.4){\widetilde{\phi}} \ar[d]_{\alpha_k} &
  M_{\lambda, t} \otimes W_t
  \ar[d]^{\alpha_k \otimes (\mathrm{\Proj}_U \circ (t\mapsto q^k))} \\
  V_{\lambda + (k-1)\rho} \ar[r]^(.4){\phi^{(k)}} & V_{\lambda + (k-1)\rho} \otimes U} 
\end{equation}
The weight-zero subspace of $W_t$ is one dimensional, which allows us to define the trace function
$\widetilde \Phi(x; q, t) \in \C(q,t)[[x_1,\dotsc,x_n]]$ of $\widetilde \phi$.  The compatibility
\eqref{eq:verma_intertwine_compat} implies the relation $\widetilde \Phi(x; q, q^k) = \Phi^{(k)}(x)$
mentioned in Theorem \ref{EKfin}.

The analogue of the Gelfand-Tsetlin basis for $M_{\lambda, t}$ is obtained by
iterating the multiplicity one decomposition over $U(\mathfrak{gl}_{n-1})$, as
in the finite-dimensional case. 

\begin{proposition}
  We have the following restriction rule for $M_{\lambda, t}$ as a module over $\C(t) \otimes U_q(\mathfrak{gl}_{n-1}) \subset \C(t) \otimes U_q(\mathfrak{gl}_n)$:
\begin{equation*}
  \prescript{}{\bigl(\C(t) \otimes U_q(\mathfrak{gl}_{n-1})\bigr)}{M_{\lambda, t}} \simeq
  \bigoplus_{\substack{\mu \in \mathcal P^{(n-1)}\\\mu \subset \lambda}} M_{\mu - (\frac{1}{2})^{(n-1)}, t}
\end{equation*}
By iterating the restriction rule above we obtain a basis for $M_{\lambda, t}$
which is indexed by chains $\lambda = \mu^{(0)} \supset \mu^{(1)} \supset
\dotsb \supset \mu^{(n-1)}$ with $\mu^{(i)} \in \mathcal P^{(n-i)}$.  We refer
to this as the Gelfand-Tsetlin basis for $M_{\lambda, t}$.
\end{proposition}

\begin{proof}
This is well known to experts.  It can be proved using the maps $\alpha_k$ and the
decomposition of $V_{\lambda + (k-1)\rho}$ over $U_q(\mathfrak{gl}_{n-1})$.
\end{proof}

\begin{theorem}
With respect to the Gelfand-Tsetlin basis of $M_{\lambda, t}$, the diagonal
coefficient of the intertwining operator $\widetilde \phi$ corresponding the
the chain $\lambda = \mu^{(0)} \supset \mu^{(1)} \supset \dotsb \supset
\mu^{(n-1)}$ is equal to:
\begin{equation*}
\begin{cases}
  \displaystyle\prod_{1 \leq i \leq n} c_{\mu^{(i-1)}, \mu^{(i)}}(q, t), & \mu^{(i)} \in \mathcal P_+ \text{ for } 0 \leq i \leq n-1\\
  0,  & \text{otherwise}
\end{cases}.
\end{equation*}
\end{theorem}

\begin{proof}
  Using the compatibility \eqref{eq:verma_intertwine_compat}, and Theorem
  \ref{qtthm}, one can see that the formula holds for $t=q^k$ when $k \in \N$
  is sufficiently large.  Since the coefficient is a rational function of $q,t$
  this determines it uniquely, and the result follows.

\end{proof}

\end{document}